\documentclass[a4paper, reqno]{amsart}

\usepackage[utf8]{inputenc}
\usepackage[T1]{fontenc}
\usepackage{unicode-math}
\usepackage{amsmath}
\usepackage{bbm}
\usepackage{amsthm}
\usepackage{amssymb}
\usepackage[margin=1.5in]{geometry}

\theoremstyle{plain}
\newtheorem{thm}{Theorem}[section]
\newtheorem{lem}[thm]{Lemma}
\newtheorem{prop}[thm]{Proposition}
\newtheorem{cor}[thm]{Corollary}

\theoremstyle{definition}
\newtheorem{rmk}[thm]{Remark}
\newtheorem{defn}[thm]{Definition}

\numberwithin{equation}{section}

\newcommand{\overbar}[1]{\mkern 1.5mu\overline{\mkern-1.5mu#1\mkern-1.5mu}\mkern 1.5mu}

\newcommand{\tif}{\ \text{if}\ }
\newcommand{\tand}{\ \text{and}\ }

\newcommand{\Orb}{\operatorname{Orb}}
\newcommand{\dOrb}{\operatorname{∂Orb}}
\newcommand{\diag}{\operatorname{diag}}
\newcommand{\Supp}{\operatorname{Supp}}
\newcommand{\Def}{\operatorname{Def}}
\newcommand{\Int}{\operatorname{Int}}
\newcommand{\Vol}{\operatorname{Vol}}

\newcommand{\Aut}{\operatorname{Aut}}

\newcommand{\Hom}{\operatorname{Hom}}
\newcommand{\End}{\operatorname{End}}

\newcommand{\Lie}{\operatorname{Lie}}

\newcommand{\Spf}{\operatorname{Spf}}

\newcommand{\id}{\mathrm{id}}

\newcommand{\len}{\operatorname{len}}
\newcommand{\Stab}{\operatorname{Stab}}

\newcommand{\tensor}{\otimes}

\newcommand{\iso}{\cong}

\newcommand{\righta}{\rightarrow}

\newcommand{\mbC}{\mathbb{C}}

\newcommand{\mbF}{\mathbb{F}}

\newcommand{\mbL}{\mathbb{L}}

\newcommand{\mbN}{\mathbb{N}}

\newcommand{\mbQ}{\mathbb{Q}}
\newcommand{\mbR}{\mathbb{R}}

\newcommand{\mbX}{\mathbb{X}}

\newcommand{\mbZ}{\mathbb{Z}}

\newcommand{\mcO}{\mathcal{O}}

\newcommand{\mfm}{\mathfrak{m}}

\begin{document}

\title{An Arithmetic Transfer Identity}
\author{Andreas Mihatsch}
\maketitle

\section{Introduction}
In \cite{zhang}, W. Zhang introduces his so-called Arithmetic Fundamental Lemma (AFL). This is a conjectural identity between certain derivatives of orbital integrals and certain intersection products in unitary Rapoport-Zink spaces, see \cite[Conjecture 2.9]{zhang}. Up to now, this conjecture has only been verified in the low dimensional cases $n=2,3$ by Zhang himself and in arbitrary dimension, but under restrictive conditions, in \cite{rapoportterstiegezhang}. The method of proof is always an explicit computation of both sides of the identity.

In the present work we restrict to the case $n=2$ to formulate and verify a variation of the AFL. The idea for this variant is due to W. Zhang and was communicated to us by M. Rapoport. We now explain our main results in detail. We will elaborate on the following definitions later in the paper.

Throughout this work we fix a prime $p\neq 2$ and a quadratic extension $E/F$ of $p$-adic fields. We denote their rings of integers be $\mcO_F⊂\mcO_E$ and fix uniformizers $π_F$ and $π_E$. Let $\mbF_q$ be the residue field of $F$. Denote the non-trivial automorphism of $E/F$ by $σ:a\mapsto \overbar{a}$ and let $η$ be the quadratic character of $F^\times$ associated to $E/F$ by class field theory. For any positive integer $s$, we denote by $\mcO_s:=\mcO_F+π_F^s\mcO_E$ the order of conductor $s$.

We embed $E^\times$ into $GL_2(E)$ via $x\mapsto \diag(x,1)$. In this way, $E^\times$ acts by conjugation on $GL_2(E)$. An element $γ∈GL_2(E)$ is called regular semi-simple if its stabilizer under this action is trivial and if its orbit is Zariski closed.

Let $S(F):= \{γ∈GL_2(E)\mid γ\overbar {γ}=1\}$ and let $S(F)_{\rm{rs}}$ denote its regular semi-simple elements. The symmetric space $S(F)$ is stable under conjugation by $F^\times$. We normalize the Haar measure on $F^\times$ such that $\Vol(\mcO^\times_F)=1$. For $γ∈S(F)_{\rm{rs}}$, for $f∈C^\infty_c(S(F))$ and for $s∈\mbC$, we define the following integrals:
\begin{align*}
\Orb_γ(f,s)\ :=&\ \int_{F^\times} f(h^{-1}γh)η(h)|h|^sdh,\\
\Orb_γ(f)\ :=&\ \Orb_γ(f,0),\\
\dOrb_γ(f)\ :=&\ \left. \frac{d}{ds}\right\vert_{s=0}\Orb_γ(f,s).
\end{align*}
These integrals are absolutely convergent since, for regular semi-simple $γ$, the intersection $F^\times γ\cap\Supp(f)$ is again compact.

Let $U(1)$ be the unitary group for a one-dimensional hermitian space for $E/F$. In particular, $U(1)(F)=\{x∈E\mid N_{E|F}(x)=1\}$. We define the two unitary groups $U_0:=U(1\oplus 1)$ and $U_1:=U(ε\oplus 1)$ where $ε∈F^\times$ is not a norm. The groups $U_0(F)$ and $U_1(F)$ are subgroups of $GL_2(E)$, stable under conjugation by $U(1)(F)$. We denote their regular semi-simple elements by $U_0(F)_{\mathrm{rs}}$ and $U_1(F)_{\mathrm{rs}}$ respectively. For $i\in \{0,1\}$, for $δ∈U_i(F)_{\rm{rs}}$ and for $φ∈C^\infty_c(U_i(F))$, we define
$$\Orb_δ(φ)=\int_{U(1)(F)} φ(h^{-1}γh)dh.$$
Here the Haar measure is normalized such that $U(1)(F)$ has volume $1$.

Two elements $γ∈S(F)_{\rm{rs}}$ and $δ∈U_i(F)_{\rm{rs}}$ are said to match if they are conjugate under $E^\times$. By \cite[Lemma 2.3]{zhang}, this relation defines a bijection of regular semi-simple orbits (of the actions of $F^\times$ and $U(1)(F)$):
$$[S(F)_{\rm{rs}}]\iso [U_0(F)_{\rm{rs}}]\sqcup [U_1(F)_{\rm{rs}}].$$

A function $f∈C^\infty_c(S(F))$ and a pair of functions $(φ_0,φ_1)∈C^\infty_c(U_0(F))\times C^\infty_c(U_1(F))$ are said to be transfers of each other if for all $γ∈S(F)_{\rm{rs}}$,
$$ω(γ)\Orb_γ(f)=\begin{cases} \Orb_δ(φ_0) & \tif γ \text{ matches } δ∈U_0(F)\\
                          \Orb_δ(φ_1) & \tif γ \text{ matches } δ∈U_1(F).\end{cases}$$
The transfer factor $ω(γ)∈\mbC^\times$ will be defined in Section 3.

Now let $\breve E/\breve F$ be the completions of the maximal unramified extensions of $E$ and $F$ with rings of integers $\mcO_{\breve F}⊂\mcO_{\breve E}$. Let $\mbF$ be their residue field and $\mbX/\mbF$ the unique formal $p$-divisible $\mcO_F$-module of height $2$ and dimension $1$. Let $D$ be the quaternion algebra over $F$ with ring of integers $\mcO_D$ and standard involution $ι:a\mapsto a^*$. We fix an isomorphism $\mcO_D\iso\End(\mbX)$ and endow $\mbX$ with an $\mcO_F$-linear principal polarization such that the Rosati involution induces the standard involution on $\mcO_D$.

In addition, we endow $\mbX$ with an action of $\mcO_E$ such that an element $a∈\mcO_E$ acts on $\Lie(\mbX)$ via the structure morphism $\mcO_E→\mbF$. This action is induced by an embedding $\mcO_E⊂\mcO_D$. We define $\overbar{\mbX}$ to be the same $p$-divisible group (with the same $\mcO_D$-action), but with the $σ$-conjugated action of $\mcO_E$.

Now let $X_i$ and $Y_j$ be two quasi-canonical lifts of $\mbX$ of levels $i$ and $j$, defined over a finite extension $A/\mcO_{\breve F}$ of ramification index $e$. We refer to \cite{wewers} for the definition and properties of such lifts. We define
\begin{equation*}\mbX^{(2)}:=\begin{cases}\mbX\times \overbar{\mbX} & \tif E/F \text{ ramified or if } i+j \text{ is even}\\
                                          \mbX\times \mbX & \tif E/F \text{ unramified and if } i+j \text{ is odd}.\end{cases}\end{equation*}
The group $\mbX^{(2)}$ is endowed with the diagonal action of $\mcO_E$ and the diagonal polarization.

Denote by $G⊂\Aut^0_{\mcO_E}(\mbX^{(2)})$ the group of $\mcO_E$-linear quasi-isogenies which preserve the polarization. Then $G⊂GL_2(D)$ is a unitary group and there is an isomorphism
$$G\iso \begin{cases}U_1(F) & \tif E/F \text{ ramified or if } i+j \text{ is even}\\
                     U_0(F) & \tif E/F \text{ unramified and if } i+j \text{ is odd},\end{cases}$$
as explained at the beginning of Section 4. We say that $γ∈S(F)_{\rm{rs}}$ matches $g∈G$, if it matches an element in the correct unitary group which maps to $g$ under this isomorphism.

For any $g∈G$, let $\Int(g)$ be the $\mcO_{\breve F}$-length of the maximal closed subscheme of $\Spf(A)$ to which $g$ deforms as automorphism of $X_i\times Y_j$.

We define $K_{i,j}:=\Stab(\mcO_i\oplus \mcO_j)\cap U_0(F)$ if $E/F$ is ramified or if $i+j$ is even. Otherwise we set $K_{i,j}:=\Stab(\mcO_i\oplus \mcO_j)\cap U_1(F)$. In either case, $1_{K_{i,j}}$ is the characteristic function of $K_{i,j}$. Then our main results are the following two theorems.

\begin{thm}\label{firstthm}
There exists a function $f∈C^\infty_c(S(F))$ which is a transfer of 
$$\begin{cases} (e\cdot 1_{K_{i,j}},0) & \tif E/F \text{ ramified or if }i+j \text{ is even}\\
                (0,e\cdot 1_{K_{i,j}}) & \tif E/F \text{ unramified and if } i+j \text{ is odd}\end{cases}$$
with the following property. For any $γ∈S(F)_{\rm{rs}}$ matching $g$ in $G$, the length $\Int(g)$ is finite and there is an equality
\begin{align}\label{firststmt}ω(γ)\dOrb_γ(f)=\Int(g)\cdot \log q.\end{align}
\end{thm}

\begin{thm} (Arithmetic transfer identity)\label{secondthm}
For any function $f∈C^\infty_c(S(F))$ which is a transfer of
$$\begin{cases} (e\cdot 1_{K_{i,j}},0) & \tif E/F \text{ ramified or if } i+j \text{ is even}\\
                (0,e\cdot 1_{K_{i,j}}) & \tif E/F \text{ unramified and if } i+j \text{ is odd,}\end{cases}$$
there exists a function $f_{\rm{corr}}∈C^\infty_c(S(F))$ with the following property. For any $γ∈S(F)_{\rm{rs}}$ matching $g$ in $G$, the length $\Int(g)$ is finite and there is an equality
\begin{align}\label{secondstmt}ω(γ)\dOrb_γ(f)=[\Int(g)+ω(γ)\Orb_γ(f_{\rm{corr}})]\cdot \log q.\end{align}
\end{thm}

These results can be considered as part of the program initiated by W. Zhang whose aim is to extend the range of the applicability of the AFL conjecture. The AFL conjecture is formulated for an unramified extension $E/F$ and ``trivial level structure'' $i=j=0$. The involved Rapoport-Zink spaces are formally smooth and the geometric side of the AFL is defined via intersection theory. Furthermore, the function on $S(F)$ whose orbital integrals are supposed to express the arithmetic intersection products in question is explicitly given and of a very simple nature. Also its transfer to the unitary side is explicitly given. 

By contrast, our extension $E/F$ is possibly ramified and $i,j\geq 0$. Then the function from Theorem \ref{firstthm} is no longer explicit. Although it is possible to write down such a function $f$ in coordinates, there is neither a natural nor a convenient choice. Note that even if $i=j=0$, we do not know a natural choice for $f$ in the ramified situation.

Note that there is a list of moduli problems where one can conjecture an arithmetic transfer identity, see \cite{RSZ}. In all these cases, the involved Rapoport-Zink spaces are regular. Our problem does not appear on this list since for $i,j>0$, there appear non-regular moduli spaces in the definition of $\Int(g)$. This is why we chose to give an ad hoc definition of $\Int(g)$ as a length.

There are two ingredients in the proof of the arithmetic transfer identity. The first is the formula of Gross and Keating for the deformation lengths of quasi-endomorphisms of quasi-canonical lifts, see \cite{gross}, \cite{grosskeating} and \cite{keating}. See also the account given in \cite{argos}. We need a slight extension of this formula as presented in \cite[Section 7]{kudlarapoport}. The second ingredient comes from harmonic analysis. More precisely, one has to know the existence of various test functions on the symmetric space $S(F)$. The case of general $n$ was solved by W. Zhang by reduction to the Lie algebra, see \cite{zhang3}. Here we give direct proofs in the case $n=2$.

The layout of this paper is as follows. In Section 2 we present a computation of the quantities $\Int(g)$ following \cite[Section 7]{kudlarapoport}. This determines the right hand side of \eqref{firststmt}. In Section 3 we give a complete characterization of functions of the form $γ\mapsto \Orb_γ(f)$ and $γ\mapsto \dOrb_γ(f)$. In Section 4 we prove the theorems stated above.

\subsection*{Acknowledgments}
I would like to thank M. Rapoport for suggesting this topic and for his continuing interest in this work. I also thank W. Zhang for helpful correspondence.

\section{Deformation of homomorphisms of quasi-canonical lifts}
In this section we first compute the space of homomorphisms between two quasi-canonical lifts. Then we derive an explicit formula for the length of the deformation locus of a quasi-homomorphism between quasi-canonical lifts. Our computations are essentially done in the work of S. Kudla and M. Rapoport, see \cite[Section 7]{kudlarapoport}. There is a mistake in their Lemma 7.4 which we correct in Lemma \ref{correct}. 

In this section we do not exclude the case $p = 2$. Apart from that, we use the notation from the introduction.

Let $M_s/\breve E$ be the ring class field associated to the order $\mcO_s$, with ring of integers $W_s$ and maximal ideal $\mfm_s$. If $s\geq 1
$, then its ramification index over $\breve F$ is
$$e_s:=[\mcO_E^\times:\mcO_s^\times]=\begin{cases}2q^s & \tif E/F \text{ is ramified}\\
                                                  q^s+q^{s-1} & \tif E/F \text{ is unramified}.\end{cases}$$
If $s=0$, then $M_s=\breve E$.

\subsection{Quasi-canonical lifts}
Let $X_0$ be the Lubin-Tate module associated to the series
$$[π_E]_0(t):=\begin{cases}π_Et+t^q & \tif E/F \text{ is ramified}\\
                        π_Et+t^{q^2} & \tif E/F \text{ is unramified}.\end{cases}$$
This is a formal $\mcO_E$-module over $\mcO_E$ in the sense of \cite{viehmannziegler}. But we consider it as a formal $\mcO_F$-module over $\mcO_{\breve E}$. Let $\mbX$ be its reduction modulo $π_E$, which is a formal $\mcO_F$-module of height $2$ over $\mbF$. The last condition means that the multiplication by any uniformizer of $\mcO_F$ has height $2$. The module $X_0$ will be called the canonical lift of $\mbX$.

By \cite[Theorem 1.1]{wewers}, we can fix an isomorphism $\End(\mbX)\iso \mcO_D$. The $\mcO_E$-action on $X_0$ induces an embedding $\mcO_E\hookrightarrow\mcO_D$ such that the action of $\mcO_E$ on $\Lie(\mbX)$ agrees with $\mcO_E→\mbF$. Composing this action with $σ$ yields $\overbar{\mbX}$ and $\overbar{X_0}$ which are formal $\mcO_F$-modules with $\mcO_E$-action.

Let $Π∈\mcO_D$ be the Frobenius $t\mapsto t^q$. If $E/F$ is ramified, then $Π=π_E$. If $E/F$ is unramified, then $Πa=\overbar{a}Π$ for all $a∈\mcO_E$ and $Π^2=π_E$. In either case, $Π$ uniformizes $\mcO_D$.

The following facts are proven in \cite{wewers}. Let $T$ be the $p$-adic Tate-module of the generic fiber of $X_0$. It is a free $\mcO_E$-module of rank one generated by $t$, say. Any $\mcO_F$-superlattice $T⊂S⊂T\tensor_{\mcO_E}E$ defines, after a finite extension $A/\mcO_{\breve E}$, a finite subgroup $S/T⊂X_0$. The quotient $X_0→X$ is a formal $\mcO_F$-module over $A$. We call $S$ minimal of level $s$, if $S=(\mcO_E+π_F^{-s}\mcO_F)a\cdot t$ for some $a∈\mcO_E^\times$.

Let $S$ be minimal of level $s$ with corresponding quotient $α_s:X_0→X_s'$. This quotient is defined over $W_s$. The isogeny $α_s$ reduces to $Π^s$ on the special fibre and so $X_s'$ is a deformation of $\mbX$. The endomorphisms of $X_s'$ can be described in two ways. The first is to consider $\End(X_s')$ as a subset of $\End(\mbX)=\mcO_D$ by the reduction of homomorphisms. This subset coincides with $\mcO_s⊂\mcO_E$. The second is as explained in \cite[Corollary 2.3]{wewers}. Here $\mcO_s⊂\End(X_0)=\mcO_E$ are the elements $ϕ$ such that there exists $ϕ'$ with $α_s\circ ϕ=ϕ'\circ α_s$. Again this induces an isomorphism $\mcO_s\iso \End(X_s')$.

These two actions of $\mcO_s$ on $X_s'$ coincide if $E/F$ is ramified or if $s$ is even. Otherwise they differ by the Galois conjugation of $E/F$. We define $X_s$ to be the formal $\mcO_F$-module $X_s'$ together with the first $\mcO_s$-action. This means that $a∈\mcO_s$ acts on $\Lie(X_s)$ via $\mcO_s⊂W_s$. The formal module $X_s$ is called a quasi-canonical lift of level $s$. The set of isomorphism classes of quasi-canonical lifts of level $s$ is a principal homogeneous space under the group $\mcO_E^\times/\mcO_s^\times$ via its action on the minimal lattices.

Finally, if $A/\mcO_{\breve E}$ is a finite extension and $X/A$ a deformation of $\mbX$ with $\End(X)=\mcO_s⊂\End(\mbX)$, then $W_s⊂A$ and $X$ is a quasi-canonical lift of level $s$.

\begin{prop}\label{propglobalhom}
Let $X_i$ and $Y_j$ be quasi-canonical lifts of level $i$ and $j$ defined over a finite extension $A/\mcO_{\breve E}$. Then there exists an $a∈\mcO_E^\times$ and an equality
$$\Hom_{\mcO_F}(X_i,Y_j)= Π^{|i-j|}\mcO_{\min\{i,j\}}\cdot a.$$
Here the left hand side embeds into $\End(\mbX)\iso \mcO_D$ by the reduction of homomorphisms.
\end{prop}
\begin{proof}
The quasi-canonical lifts $X_i$ and $Y_j$ are defined by minimal superlattices of $T$ as explained above. They take the form
$$S_i = (π^{-i}\mcO_F+\mcO_E)a_i\cdot t\ \tand\ S_j = (π^{-j}\mcO_F+\mcO_E)a_j\cdot t,$$
where $a_i,a_j ∈ \mcO_E^\times$ are chosen suitably.

Let $α_i:X_0→X_i$ and $α_j:X_0→Y_j$ be the corresponding isogenies. According to \cite[Corollary 2.3]{wewers}, we have
$$α_j^{-1}\circ \Hom_{\mcO_F}(X_i,Y_j)\circ α_i=\{x∈\mcO_E\mid xS_i⊂S_j\}.$$
With $\tilde{a}=a_i/a_j$, we get
$$\{x∈\mcO_E\mid xS_i⊂S_j\}=\begin{cases}\mcO_{\min\{i,j\}}\tilde{a} & \tif i<j\\
                                     π_F^{i-j}\mcO_{\min\{i,j\}}\tilde{a} & \tif i\geq j.\end{cases}$$
Multiplying with $Π^{-i}$ on the right and $Π^j$ on the left yields the result with $a=σ^i(\tilde a)$.
\end{proof}

\begin{cor}\label{corglobalhom}
Let $X_i$ and $Y_j$ be quasi-canonical lifts of level $i$ and $j$. Then\smallskip\\
\emph{a)} $\Hom_{\mcO_F}(X_i,Y_j)$ is a free $\mcO_{\min\{i,j\}}$-module of rank $1$ where scalar multiplication is given by composition.\smallskip\\
\emph{b)} The reduction of a homomorphism $X_i→Y_j$ commutes with the $\mcO_E$-action on $\mbX$ if and only if $E/F$ is ramified or if $i+j$ is even. Otherwise it Galois-commute with the $\mcO_E$-action.$\hfill\qed$
\end{cor}

Recall \cite[Proposition 4.6]{wewers}, stating that the modulus of $X_s$ uniformizes $W_s$ if $s\geq 1$. This means the following. After the choice of a formal coordinate $X_s\iso \Spf W_s[[t]]$, we can write the multiplication by $π_F$ as a power series
$$[π_F]_s(t)=π_Ft+\ldots+u_st^q+\ldots.$$
Then $u_s$ uniformizes $W_s$. In particular, $[π_F]_s$ has height $1$ modulo $\mfm_s^2$. This implies the next lemma.

\begin{lem}\label{correct}
Consider two quasi-canonical lifts $X_i$ and $Y_j$ of levels $i>j$ defined over a finite extension $A/W_i$ of ramification index $e$. Let $\mfm⊂A$ be the maximal ideal. Then any automorphism of $\mbX$ lifts to an isomorphism $X_i\tensor A/\mfm^e\iso Y_j\tensor A/\mfm^e$. But $X_i\tensor A/\mfm^{e+1}$ and $X_j\tensor A/\mfm^{e+1}$ are not isomorphic as $\mcO_F$-modules.
\end{lem}
\begin{proof}
According to \cite[Theorem 3.8]{viehmannziegler}, the universal deformation space of $\mbX$ is $\Spf \mcO_{\breve E}[[x]]$. Let $φ_i,φ_j:\mcO_{\breve E}[[x]]\righta A$ be the homomorphisms corresponding to $X_i$ and $Y_j$. According to \cite[Lemma 3.5]{viehmannziegler}, $φ_i(x)=u_i$ and $φ_j(x)=u_j$. Now $u_i\equiv u_j$ modulo $\mfm^e$, but not modulo $\mfm^{e+1}$. It follows that $X_i\tensor A/\mfm^e\iso Y_j\tensor A/\mfm^e$ and $X_i\tensor A/\mfm^{e+1}\ncong X_s\tensor A/\mfm^{e+1}$.

We still have to show that any automorphism $α$ of $\mbX$ lifts to $X_i\tensor A/\mfm^e$. But $α$ induces an automorphism $α^*$ of the universal deformation space and $(φ_i\circ α^*)(x)\equiv φ_i(x) \mod \mfm^e$.
\end{proof}

\subsection{Deformation of homomorphisms}
For the rest of this section we fix two quasi-canonical lifts $X_i$ and $Y_j$ of levels $i$ and $j$ defined over some finite extension $A/\mcO_{\breve F}$. In particular $W_i,W_j⊂A$ and we let $e$ be the ramification index of $A$ over $W_{\max\{i,j\}}$. Let $\mfm⊂A$ be the maximal ideal and set $d=|j-i|$.

We now compute the spaces $H_n:=\Hom_{\mcO_F}(X_i\tensor A/\mfm^{n+1},Y_j\tensor A/\mfm^{n+1})$. We also abbreviate $H_\infty:=\Hom_{\mcO_F}(X_i,Y_j)$. Let $V_l⊂H_0=\mcO_D$ denote the set of homomorphisms of height $\geq l$ and define $a(n)=1+q+\ldots+q^n$.

\begin{thm}\label{thm1}
Let $f_0∈(H_\infty + V_l)\setminus(H_\infty + V_{l+1})$ and define $n:=\lfloor(l+d)/2\rfloor$. Then $f_0$ lifts to $H_α$ but not to $H_{α+1}$ with $α+1=$
$$e\cdot    \begin{cases}   a(l)                   & \tif l<d                                              \vspace{0.4em}\\
                            a(n)+a(n-1)-a(d-1)     & \tif d\leq l \leq i+j-1 \tand l+d \text{ even}      \vspace{0.4em}\\
                            2a(n)-a(d-1)           & \tif d\leq l \leq i+j-1 \tand l+d \text{ odd}   \vspace{0.4em}\\
                            2a(j-1)-a(d-1) + \frac{l-(i+j-1)}{2}\cdot e_{\max\{i,j\}} & \tif i+j \leq l.  \vspace{0.4em}
\end{cases}$$
\end{thm}
\begin{rmk}
The fraction appearing in the theorem is always an integer. Namely $e_{\max\{i,j\}}$ is even except for the case $i=j=0$ and $E/F$ unramified. But if $E/F$ is unramified, then the integer $l-(i+j-1)$ in the last case is always even. 
\end{rmk}
\begin{proof}
First we reduce to the case $i\leq j$. The formal group $\mbX$ and the quasi-canonical lifts are also $p$-divisible groups. It is well known that the dual $p$-divisible $\mcO_F$-module $\mbX^\vee$ is isomorphic to $\mbX$. (Here $\mbX^\vee = \Hom(\mbX,\mbL)$ where $\mbL$ is a Lubin-Tate module for $F$.) It follows that $X_i^\vee$ and $Y_j^\vee$ are formal groups deforming $\mbX$ with endomorphism ring equal to $\mcO_i$ and $\mcO_j$. In particular, $X_i^\vee$ and $X_j^\vee$ are again quasi-canonical lifts of level $i$ and $j$.

Now dualizing yields a bijection of $\mcO_F$-modules,
$$\Hom(X_i\tensor W/\mfm^{n+1},Y_j\tensor W/\mfm^{n+1})\iso\Hom(Y_j^\vee\tensor W/\mfm^{n+1},X_i^\vee\tensor W/\mfm^{n+1}).$$
Its inverse is also given by dualization. This bijection commutes with the reduction of morphisms and so preserves the deformation lengths of homomorphisms. It also preserves the height and, in particular, preserves the spaces $(H_\infty+V_l)\setminus(H_\infty+V_{l+1})$. Here we defined $H_n,H_\infty$ and $V_l$ in the obvious way for the right hand side. So we can assume that $i\leq j$ from now on.

We can write $f_0=h+g_0$ with $h∈H_\infty$ and $g_0∈V_l$. It is clear that the deformations of $f_0$ and $g_0$ are in bijection via addition or subtraction of $h$.

So we can assume that $f_0$ has height $l$. Let $X_i$ be defined by the lattice $(\mcO_E+π^{-i}\mcO_F)a_i\cdot t$ and define $Y'_j$ by the lattice $(\mcO_E+π^{-j}\mcO_F)a_i\cdot t$. By Proposition \ref{propglobalhom}, there exists an $a∈\mcO_E^\times$ which lifts to an isomorphism $Y_j\iso Y'_j$. Then left multiplication with $a$ induces bijections $H_\infty\iso\Hom(X_i,Y'_j)$ and $H_n\iso \Hom_{\mcO_F}(X_i\tensor A/\mfm^{n+1},Y'_j\tensor A/\mfm^{n+1})$ and preserves the height. So it is enough to prove the theorem for $Y_j=Y'_j$.

For $i\leq k\leq j$, we define $Z_k$ to be the quasi-canonical lift associated to $(\mcO_E+π^{-k}\mcO_F)a_i\cdot t$. For any $g_0∈\mcO_D$ we define $n_{k}(g_0)$ to be the maximal $n$ (or $\infty$) such that $g_0$ lifts to a homomorphism $X_i\tensor A/\mfm^n→Z_k\tensor A/\mfm^n$. We recall \cite[Lemma 3.6]{rapoport}:

\begin{lem} \label{rapoport}
Suppose that $f_0∈\mcO_D\setminus \Hom_{\mcO_F}(X_i,Z_k)$. Then
$$n_{k+1}(Πf_0)=n_k(f_0)+e/e_{k+1}.\vspace{-7mm} $$
$\hfill\qed$
\end{lem}

\emph{Case $l<d$:} Let us assume that $l<d$. Recall from Lemma \ref{correct} that if $l=0$ and $i\neq j$, then $n_j(f_0)=e$. If $l\neq 0$, we can write $f_0=Π^lg_0$ with $g_0$ of height $0$. Then $n_{j-l}(g_0)=e\cdot e_j/e_{j-l} = e\cdot q^l$ and an inductive application of the previous lemma shows
$$n_j(f_0)=e\cdot(q^l+q^{l-1}+\ldots+1),$$
which proves Theorem \ref{thm1} in the first case.

\emph{Remaining cases:} Now assume that $l\geq d$ and write $f_0=Π^dg_0$. Then $$g_0∈(\End(X_i)+V_{l-d})\setminus(\End(X_i)+V_{l-d+1})$$ and by \cite[Theorem 2.1]{vollaard}, we have
$$n_{i}(g_0)=e\cdot e_j/e_i\begin{cases}a((l-d)/2)+a((l-d)/2-1) & \tif l-d < 2i \text{ is even}\\
                               2a((l-d)/2)         & \tif l-d < 2i \text{ is odd}\\
                               2a(i-1)+\frac{1}{2}(l-(i+j-1))e_i & \tif l-d \geq 2i.\end{cases}$$
Again we apply Lemma \ref{rapoport} $d$ times, distinguishing two cases. If $i=0$, then only the fourth case of Theorem \ref{thm1} and the third of the above cases occurs. The result is immediate. If $i>0$, then one uses $e_j=q^de_i$ to verify the formula.
\end{proof}

\section{Analytic theory}
In this section, we prove some general results about orbital integrals on $S(F)$. Our main results are Corollary \ref{charactorbints} and Corollary \ref{charactderivorbints} which characterize the functions of the form $γ\mapsto \Orb_γ(f)$ and $γ\mapsto \dOrb_γ(f)$. Another important result is Theorem \ref{existtransfer} which plays a key role in the proof of the arithmetic transfer identity. Theorem \ref{existtransfer} has been proven by W. Zhang for any $n$, see \cite[Theorem 2.6]{zhang3}.

We take up all notation from the introduction.

\subsection{Orbital integrals on $S(F)$}\label{orbints}
First note that an element
$$γ:=\left(\begin{array}{cc}
a & b\\
c & d
\end{array}\right)∈GL_2(E)$$
is regular semi-simple if and only if $b\neq 0$ and $c\neq 0$. We denote its entries by $a(γ),b(γ),c(γ)$ and $d(γ)$.

Parametrizations of $S(F)_{\rm{rs}}$ and of $B:=S(F)\setminus S(F)_{\rm{rs}}$ are given as follows:
\begin{equation}\begin{aligned}\label{SF}
S(F)_{\rm{rs}}= & \left\{\left(\begin{array}{cc}
a & b\\
(1-N(a))/\overbar{b} & -\overbar {a}b/{\overbar b}
\end{array}\right)\ \text{with}\ \ b\neq 0\ \text{and}\ 1-N(a)\neq 0\right\},\\
B= & \left\{\left.\left(\begin{array}{cc} a & 0 \\ c & d\end{array}\right) \ \right |\ a,d∈U(1)(F) \tand c\overbar{a}+d\overbar{c}=0\right\}\\
\cup & \left\{\left.\left(\begin{array}{cc} a & b \\ 0 & d\end{array}\right) \ \right |\ a,d∈U(1)(F) \tand a\overbar{b}+b\overbar{d}=0\right\}.
\end{aligned}\end{equation}
Denote by $B_0⊂B$ the diagonal matrices, which we identify as $B_0\iso U(1)(F)\times U(1)(F)$ by $γ\mapsto (a(γ),d(γ))$.

\begin{rmk}\label{rmkB}
In general it is easy to understand locally constant functions on $S(F)_{\rm{rs}}$ and their orbital integrals. The subtle point is to ensure that certain functions extend as locally constant functions to all of $S(F)$. The set $B_0$ will play a crucial role in this context. Namely any orbit passing close to $γ∈B$ will also pass close to $B_0$. So any $F^\times$-invariant function $f$ on $S(F)_{\rm{rs}}$ is determined in a neighborhood of $B$ by its behavior near $B_0$.
\end{rmk}

\begin{rmk}\label{rmkTop}
Any function $f∈C^\infty_c(S(F))$ has an extension $\tilde{f} ∈C^\infty_c(GL_2(E))$. It is also possible to define $\Orb_γ(\tilde f)$ for any regular semi-simple $γ∈GL_2(E)$. So we will often formulate topological statements for the group $GL_2(E)$, leaving the restriction to $S(F)$ implicit.
\end{rmk}

The integral $\Orb_γ(f,s)$ transforms by the character $η_s(\cdot):=η(\cdot)|\cdot|^s$ under the action of $F^\times$. More precisely,
$$\Orb_γ(λ^*f,s)=η_s^{-1}(λ)\Orb_γ(f,s)=\Orb_{λ^{-1}γλ}(f,s).$$
Differentiating yields
\begin{equation}\label{invariance}\dOrb_γ(λ^*f)=η(λ)\left[\dOrb_γ(f)-\log|λ|\Orb_γ(f)\right]=\dOrb_{λ^{-1}γλ}(f).\end{equation}
If $χ$ is any character of $F^\times$, then we call a function $\phi$ on $S(F)_{\rm{rs}}$ to be $χ$-invariant if $\phi(λ^{-1}γλ)=χ(λ)\phi(γ)$ for all $λ∈F^\times,γ∈S(F)_{\rm{rs}}$. For example, $γ\mapsto \Orb_γ(f,s)$ is $η_s^{-1}$-invariant for any test function $f∈C^\infty_c(S(F))$.

\begin{lem}\label{basicsonorbitalintegrals}
Let $f∈C^\infty_c(S(F))$.\smallskip\\
\emph{a)} $γ\mapsto \dOrb_γ(f)$ is η-invariant if and only if all orbital integrals $\Orb_γ(f)$ vanish.\smallskip\\
\emph{b)} There exists $f'∈C^\infty_c(S(F))$ such that $\Orb_γ(f)=\dOrb_γ(f')$ for all $γ∈S(F)_{\rm{rs}}$.
\end{lem}
\begin{proof}
The first assertion follows directly from the transformation behavior of $\dOrb$ in \eqref{invariance}. For the second assertion we compute
$$\dOrb_γ(η(λ)f-λ^*f)=η(λ)\log|λ|\Orb_γ(f).$$
Then we can choose $λ∈F^\times$ with $|λ|\neq 1$ and define $f':=(f-η(λ)λ^*f)/(\log|λ|)$.
\end{proof}

From now on, we fix an extension of $η$ to a smooth character on $E^\times$ (not necessarily quadratic). It will again be denoted by $η$. This defines an extension of $η_s$, again denoted by $η_s$, to $E^\times$. Namely we set $η_s=η|\cdot|^s$ where $|\cdot|$ denotes the extension of the absolute value of $F$ to $E$.

In light of Remark \ref{rmkB}, the following proposition is very useful.
\begin{prop}\label{propfunctions}
Given $f∈C^\infty_c(S(F))$, there exists $f'∈C^\infty_c(S(F))$ such that for all $γ∈S(F)_{\rm{rs}}$,
\begin{align*}\Orb_γ(f) & =\Orb_γ(f'),\\
\dOrb_γ(f) & =\dOrb_γ(f')\end{align*}
and $f'|_{B_0}=0$.
\end{prop}
\begin{proof}
By definition $f|_{B_0}$ is locally constant. Let
$$\mcO_E^\times \times \mcO_E^\times = \coprod V_a^i\times V_d^i$$
be a finite disjoint open covering such that $f|_{(V_a^i\times V_d^i)\cap B_0} \equiv r_i$ is constant.

For two open compact subsets $V_a,V_d⊂E$, let $1(V_a,V_d)$ be the characteristic function of the set
$$ K(V_a,V_d):=
\left\{\left(\begin{array}{cc}
V_a & \mcO_E\\
\mcO_E & V_d
\end{array}\right)\right\}\cap S(F).$$
Note that this is a compact open subset of $S(F)$ and hence $1(V_a,V_d)∈C^\infty_c(S(F))$.

We choose $λ_0∈F^\times$ with $η(λ_0)=-1$ and $v_F(λ_0)\geq 1$ to define
$$α'(V_a,V_d):=1(V_a,V_d)+λ_0^*1(V_a,V_d).$$
By the $η$-invariance of $\Orb_γ(f)$, we see that $\Orb_γ(α'(V_a,V_d))=0$ for all $γ∈S(F)_{\rm{rs}}$.

Again we choose $λ_1∈F^\times$ with $η(λ_1)=-1$ and $v_F(λ_1)\geq 1$ to define
$$α(V_a,V_d):=\frac{1}{4}\Big(α'(V_a,V_d)+λ_1^*α'(V_a,V_d)\Big).$$
Then by Lemma \ref{basicsonorbitalintegrals} a), $\Orb_γ(α(V_a,V_d))=\dOrb_γ(α(V_a,V_d))=0$ for all $γ∈S(F)_{\rm{rs}}$.

It follows that
$$f':=f-\sum_i r_iα(V_a^i, V_d^i)$$
satisfies the conditions of the proposition.
\end{proof}

\subsection{Germ expansion of orbital integrals}\label{germs}

In the following, $X$ denotes the space $\mcO_E^\times\times\mcO_E^\times\times\Big(E^\times/F^\times\Big)$. It contains a parameter space for the $F^\times$-orbits of each of the two components of $B$ in \eqref{SF}.

\begin{thm}\label{thmorbints}
Let $f∈C^\infty_c(S(F))$ such that $f|_{B_0}=0$. Then there exist two locally constant functions
$$A_0,A_1:X→\mbC[q^s,q^{-s}]$$
such that for all $γ=\left(\begin{smallmatrix} a & b\\ c& d\end{smallmatrix}\right)∈S(F)_{\rm{rs}}$ near $B_0$,
\begin{equation}\label{germexpansion}
\Orb_γ(f,s)=η_s(b)A_0(s;a,d,b)+η_s(c)^{-1}A_1(s;a,d,c).
\end{equation}

Conversely, given locally constant $A_0,A_1$ as above, there exists an $f∈C^\infty_c(S(F))$ such that identity \eqref{germexpansion} holds.
\end{thm}
\begin{defn}
We call identity \eqref{germexpansion} the germ expansion of $f$. This terminology is taken from \cite{zhang2}. We also write $A_0(γ)$ and $A_1(γ)$ instead of $A_0(s;a,d,b)$ and $A_1(s;a,d,c)$. 
\end{defn}
\begin{proof}
Let $f∈C^\infty_c(S(F))$ with $f|_{B_0}=0$.\smallskip\\
\emph{Step 1:}
For $(a,d,b),(a,d,c)∈X$, we define
$$A_0(s;a,d,b):=\frac{1}{η_s(b)}\int_{F^\times} η_s(h) f\Big(\Big(\begin{smallmatrix} a & b/h\\0& d\end{smallmatrix}\Big)\Big)\ dh,$$
$$A_1(s;a,d,c):=η_s(c)\int_{F^\times} η_s(h) f\Big(\Big(\begin{smallmatrix} a & 0\\ch& d\end{smallmatrix}\Big)\Big)\ dh.$$
Note that these integrals are absolutely convergent, since the integrand has no support near $B_0$. They are $F^\times$-invariant in $b$ and $c$ and hence well defined.

\emph{Step 2:}
Now fix $(a,d)∈B_0$ and choose a neighborhood $V_a\times V_d⊂\mcO_E^\times\times\mcO_E^\times$ such that for all
$$γ∈\left\{\left(\begin{smallmatrix}V_a & *\\ * & V_d\end{smallmatrix}\right)\right\},$$
the value $f(γ)$ is independent of $a(γ)$ and $d(γ)$. Such a neighborhood exists, since $f$ is locally constant with compact support.

There exists an integer $N$ and two finite families of elements $c_i,b_j$ with $i=1,\ldots,n$ and $j=1,\ldots,m$ with the following properties. The open sets (in $S(F)$)

$$T_0:=\left\{\left(\begin{array}{cc} V_a & π^N\mcO_E \\ π^N\mcO_E & V_d \end{array}\right)\right\},\ 
T_i:=\left\{\left(\begin{array}{cc} V_a & π^N\mcO_E \\ c_i + π^N\mcO_E & V_d \end{array}\right)\right\},$$

$$S_j:=\left\{\left(\begin{array}{cc} V_a & b_j + π^N\mcO_E \\ π^N\mcO_E & V_d \end{array}\right)\right\}$$

are disjoint, cover $\Supp f\cap B$ and $f$ is constant on each of them. Let $Z$ be their union.

Then there exists a neighborhood $W_0$ of $\diag(a,d)$ such that for any $γ∈W_0$, we have $F^\times γ\cap\Supp(f)⊂Z$. Note that $f\vert_{T_0}\equiv 0$, so the maps $S(F)_{\rm{rs}}\ni γ\mapsto \Vol(F^\times γ \cap T_i)$ and $γ\mapsto \Vol(F^\times γ \cap S_j)$ are locally constant in a neighborhood $W_1$ of $\diag(a,d)$. Since they are clearly $F^\times$-invariant, they yield locally constant functions on $X$. It is now clear, that the theorem holds for $γ∈W_0\cap W_1$. Since $(a,d)∈B_0$ was arbitrary, this finishes the proof of the first part of the theorem.

\emph{Proof of the second part:}
Now let $A_0,A_1$ be given. We want to construct a suitable function $f∈S(F)_{\rm{rs}}$. By linearity and symmetry of the argument we assume that $A_1=0$ and that $A_0$ takes values in $q^{ks}\mbC⊂ℂ[q^s,q^{-s}]$ for some $k∈\mbZ$.

\emph{Reduction to $k=0$:} Assume for the moment the existence of $f$ whenever $A_0$ takes values in $q^0ℂ$. Let $f_0$ be a function with germ expansion associated to $q^{-ks}A_0$ and $A_1=0$. Choose $λ_0∈F^\times$ with valuation $v_F(λ_0)=k$. Then by the $η^{-1}_s$-invariance, $η(λ_0)λ_0^*f_0$ has germ expansion associated to $A_0$ and $A_1$.

\emph{Case $k=0$:} So we can assume that $A_0$ takes values in $\mbC$. Let $K⊂E^\times$ be a compact open subset such that $F^\times \cdot K=E^\times$ and let $β$ be the characteristic function of the set
$$\left\{\left(\begin{matrix} \mcO_E^\times & K \\ \mcO_E & \mcO_E^\times\end{matrix}\right)\right\}.$$
For $γ=\left(\begin{smallmatrix} a&b\\c&d\end{smallmatrix}\right)∈S(F)$ we define
$$f(γ):=\begin{cases} 0 & \tif\ F^\times b\cap K = \emptyset\\
                    β(γ) η_s(b) \Vol(F^\times b\cap K)^{-1} A_0(a,d,b) & \rm{otherwise}.\end{cases}$$

Clearly this function is locally constant with compact support. We compute the germ expansion of $f$ as in the first part of the theorem. For this we assume that $γ$ is near $B_0$.
\begin{align*}
&\ η_s(b)^{-1}\int_{F^\times} η_s(h) f\Big(\Big(\begin{smallmatrix} a & b/h\\0& d\end{smallmatrix}\Big)\Big)\ dh\\
= &\ \frac{η_s(b)^{-1}}{\Vol(F^\times b\cap K)} \int_{F^\times b\cap K} η_s(h)η_s(b/h) A_0(a,d,b/h) dh\\
= &\ \frac{1}{\Vol(F^\times b\cap K)} \int_{F^\times b\cap K} A_0(a,d,b) dh\\
= &\ A_0(a,d,b)\end{align*}
In the second to last equality, we used that $A_0$ is invariant under multiplication with $h∈F^\times$. The integral $A_1$ vanishes since $0\notin K$. This concludes the proof of the theorem.
\end{proof}

\begin{rmk}
The fact that $f$ can locally near $B_0$ be defined by polynomials in $q^s,q^{-s}$ which transform with $η_s$ is equivalent to $f|_{B_0}=0$. Namely if $f(\diag(a,d))\neq 0$, then the number of monomials $q^{ks}$ in $\Orb_γ(f,s)$ is not bounded for $γ$ approaching $\diag(a,d)$.
\end{rmk}

\begin{cor}\label{charactorbints}
Let $f∈C^\infty_c(S(F))$ and fix $s∈ℂ$. Then the function (on $S(F)_{\rm{rs}}$) $ϕ:γ\mapsto \Orb_γ(f,s)$ is $η_s^{-1}$-invariant and $\overbar{\Supp(ϕ)}/F^\times$ is compact. There exist locally constant functions $A_0,A_1:X→\mbC$ such that for all regular semi-simple $γ$ near $B_0$,
$$ϕ(γ)=η_s(b)A_0(γ)+η_s(c)^{-1}A_1(γ).$$

Conversely if $ϕ∈C^\infty(S(F)_{\rm{rs}})$ satisfies the above conditions, then there exists a function $f∈C^\infty_c(S(F))$ such that $\Orb_γ(f,s)=ϕ(γ)$ for all $γ∈S(F)_{\rm{rs}}$.
\end{cor}
\begin{proof}
Given $f$, we can assume that $f|_{B_0}=0$ by similar arguments as in the proof of Proposition \ref{propfunctions}. Then we apply the first part of Theorem \ref{thmorbints} and evaluate the functions $A_0$ and $A_1$ in the fixed value $s$.

Let now $ϕ$ be given. By the second part of Theorem \ref{thmorbints}, there exists a function $f_0$ such that $\Orb_γ(f_0,s)=ϕ(γ)$ in a neighborhood of $B_0$. By considering the difference $\Orb_γ(f_0,s)-ϕ(γ)$, we can assume that $ϕ=0$ in a neighborhood of $B_0$. In particular $\overbar{\Supp(ϕ)}=\Supp(ϕ)$.

Let $K⊂S(F)_{\rm{rs}}$ be open and compact, such that $F^\times K= \Supp(ϕ)$. Then $μ(γ):=\Vol(F^\times γ\cap K)$ is an $F^\times$-invariant smooth function on $S(F)_{\rm{rs}}$ which vanishes in a neighborhood of $B$. Note that $ϕ(h^{-1}γh)=h^*ϕ(γ)=η_s(h)^{-1}ϕ(γ)$ and set $f(γ):=1_K(γ)ϕ(γ)/μ(γ)$. (If $γ\notin F^\times K$, then we define $f(γ)=0$.) Then
$$\Orb_γ(f,s)=\int_{F^\times} η_s(h)1_K(h^{-1}γh)\frac{ϕ(h^{-1}γh)}{μ(h^{-1}γh)}dh=\int_{F^\times γ\cap K} \frac{ϕ(γ)}{μ(γ)}dh=ϕ(γ).$$
\end{proof}

Let $v$ be the extension of the normalized valuation from $F$ to $E$. 
\begin{cor}\label{charactderivorbints}
Let $f∈C^\infty_c(S(F))$. Then there exist locally constant functions
$$A_0,A'_0,A_1,A'_1:X→\mbC$$
such that for all regular semi-simple $γ=\left(\begin{smallmatrix}a & b\\ c& d\end{smallmatrix}\right)$ near $B_0$, there is an identity
$$\dOrb_γ(f)=η(b)\Big[v(b)A_0(a,d,b)+A_0'(a,d,b)\Big]+η(c)^{-1}\Big[v(c)A_1(a,d,c)+A'_1(a,d,c)\Big].$$

Conversely given $A_0,A_0',A_1,A_1'$ as above, there exists a function $f∈C^\infty_c(S(F))$ such that $γ\mapsto \dOrb_γ(f)$ satisfies the above identity near $B_0$.
\end{cor}
\begin{proof}
By Proposition \ref{propfunctions}, we can assume that $f|_{B_0}=0$. Let
$$\Orb_γ(f)=η_s(b)C_0(s;a,d,b)+η_s(c)^{-1}C_1(s;a,d,c)$$
be the germ expansion of $f$ from Theorem \ref{thmorbints}. Its derivative in $s=0$ is given as
\begin{equation}\begin{aligned}\label{iddOrb}
\dOrb_γ(f)= &\ η(b) \Big( C_0'(0;a,d,b) - v(b)\log(q)C_0(0;a,d,b)\Big)\\
 + &\ η(c)^{-1}\Big( C_1'(0;a,d,c) + v(c)\log(q)C_1(0;a,d,c)\Big).
\end{aligned}\end{equation}
We now perform the obvious substitutions.

Conversely if $A_0,A_0',A_1,A_1'$ are given, then we choose families of polynomials $C_0(s),C_1(s):X→\mbC[q^s,q^{-s}]$ having the following values and derivatives in $s=0$:
$$C_0'(0;γ)=A_0'(γ)\tand -\log(q)C_0(0;γ) = A_0(γ),$$
$$C_1'(0;γ)=A_1'(γ)\tand \log(q)C_1(0;γ) = A_1(γ).$$
Then we apply the second part of Theorem \ref{thmorbints}.
\end{proof}

\begin{cor}\label{maincor}
Let $f∈C^\infty_c(S(F))$ be such that $\Orb_γ(f)=0$ for all $γ∈S(F)_{\rm{rs}}$. Then there exists a function  $f'∈C^\infty_c(S(F))$ such that for all $γ∈S(F)_{\rm{rs}}$,
$$\dOrb_γ(f)=\Orb_γ(f').$$
\end{cor}
\begin{proof}
By Corollary \ref{charactorbints} it is enough to show that $A_0$ and $A_1$ in Corollary \ref{charactderivorbints} vanish. Let us consider the germ expansion of $f$,
$$\Orb_γ(f)=η(b(γ))C_0(γ)+η(c(γ))^{-1}C_1(γ).$$ 
It follows from formula \eqref{iddOrb}, that $A_0=-\log(q)C_0$ and $A_1=\log(q)C_1$. Hence we need to show that $C_0=C_1=0$.

Let $λ∈F^\times$ with $η(λ)=-1$. Then by the assumption on $f$, for any $γ$ close enough to $B_0$,
$$η(b(γ))C_0(γ)+η(c(γ))^{-1}C_1(γ)=0$$
and
$$η(b(γ))C_0(γ)+η(λc(γ))^{-1}C_1(λ^{-1}γλ)=0.$$
Since $C_1(λ^{-1}γλ)=C_1(γ)$, it follows that $C_0=C_1=0$.
\end{proof}

\begin{rmk}
Let $χ:E^\times→\mbC^\times$ be a smooth character or let $χ_s$ be a family of such characters. Then all the statements in this subsection should have analogues for a $χ$-twisted orbital integral $\Orb_γ(f,χ)$ and the family $\Orb_γ(f,χ_s)$. They should also hold for orbital integrals on $GL_2(E)$, since we never used the structure of $S(F)$ in our proofs.
\end{rmk}

\subsection{Transfer of Functions}\label{transfer}
For $ε∈F^\times$, we let $U^ε:=U(ε\oplus 1)$ be the corresponding unitary group. Then $U^ε(F)⊂GL_2(E)$ is stable under the conjugation by $U(1)(F)$. For $δ∈U^ε(F)_{\rm{rs}}$ and $φ∈C^\infty_c(U^ε(F))$, we define
$$\Orb_δ(φ)=\int_{U(1)(F)} φ(h^{-1}γh)dh.$$

Elements $γ∈S(F)_{\rm{rs}}$ and $δ∈U^ε(F)_{\rm{rs}}$ are said to match, if they are conjugate under $E^\times$. A direct computation shows that a given γ matches some $δ∈U^ε(F)$ if and only if $(1-N(a))/ε$ is a norm of $E/F$. Assume this is the case and let $N(x)=(1-N(a))/ε$. Then
\begin{align}\label{match}γ\ \text{matches } δ=\left(\begin{array}{cc}
a & x\\
\overbar{x}bε/\overbar b & -\overbar ab/\overbar b
\end{array}\right)∈U^ε(F).\end{align}
Conversely, a given $δ∈U^ε(F)_{\rm{rs}}$ has a match in $S(F)$.

As in the introduction, we let $U_0=U^{ε_0}$ and $U_1=U^{ε_1}$ be the unitary groups associated to a norm $ε_0$ and a non-norm $ε_1$. Then the matching relation defines a bijection of conjugation orbits:
$$[S(F)_{\rm{rs}}]\iso [U_0(F)_{\rm{rs}}]\sqcup [U_1(F)_{\rm{rs}}].$$

For a given function $f∈C^\infty_c(S(F))$, the orbital integral $γ\mapsto \Orb_γ(f)$ is not $F^\times$-invariant and hence cannot descend to the quotient $[S(F)_{\rm{rs}}]$. This motivates the following definition.

\begin{defn}
We define the transfer factor $ω$ by the formula
\begin{equation}\label{trafa}ω(γ):=η(c(γ)).\end{equation}
This is a smooth $η$-invariant function on $S(F)_{\rm{rs}}$.
\end{defn}

\begin{defn}\label{deftransfer}
We say that $f∈C^\infty_c(S(F))$ and $(φ_0,φ_1)∈C^\infty_c(U_0(F))\times C^\infty_c(U_1(F))$ are transfers of each other if for each $γ∈S(F)_{\rm{rs}}$ matching $δ∈U_i(F)$, there is an equality
$$ω(γ)\Orb_γ(f)=\Orb_δ(g_i).$$

In particular, $f∈C^\infty_c(S(F))$ is a transfer of $(0,0)$ if and only if $\Orb_γ(f)=0$ for all $γ∈S(F)_{\rm{rs}}$.
\end{defn}

\begin{rmk}\label{robustness}
Note that if $h∈E^\times$, then conjugation by $h$ induces an isomorphism $U^ε→U^{N(h)ε}$. If $γ∈S(F)_{\rm{rs}}$ matches $δ∈U^ε(F)$, then $γ$ matches $h^{-1}δh∈U^{N(h)ε}(F)$. In particular, the pullback $h^*:C^\infty_c(U^{N(h)ε}(F))→C^\infty_c(U^ε(F))$ is an isomorphism which preserves the transfer in an obvious sense. It follows that the choice of $ε_0$ and $ε_1$ is irrelevant in this analytic setup.
\end{rmk}

The following theorem was already proven by W. Zhang in much greater generality, see \cite[Theorem 2.6]{zhang3}.

\begin{thm}\label{existtransfer}
Given a pair of functions $(φ_0,φ_1)∈C^\infty_c(U_0(F))\times C^\infty_c(U_1(F))$, there exists a transfer $f∈C^\infty_c(S(F))$.

Conversely any function $f∈C^\infty_c(S(F))$ has a transfer $(φ_0,φ_1)$.
\end{thm}
\begin{proof}
Consider a pair of functions  $(φ_0,φ_1)∈C^\infty_c(U_0(F))\times C^\infty_c(U_1(F))$ and the map
$$ϕ:γ\mapsto\begin{cases} ω(γ)^{-1}\Orb_δ(φ_0) & \tif γ \text{ matches } δ∈U_0(F)\\
                          ω(γ)^{-1}\Orb_δ(φ_1) & \tif γ \text{ matches } δ∈U_1(F).\end{cases}$$
It is $η$-invariant and $\overbar{\Supp(ϕ)}/F^\times$ is compact. According to Corollary \ref{charactorbints} it is enough to show that $ϕ$ has a germ expansion.

Let $i∈\{0,1\}$ and fix a diagonal matrix $\diag(a,d)∈U_i(F)$. The open neighborhoods $\diag(a,d)+M_2(π^N\mcO_E)$ are stable under conjugation by $U(1)(F)$. Since $φ_i$ is locally constant, it follows that $\Orb_δ(φ_i)$ is constant near $\diag(a,d)$ with value $C_i(a,d):=φ_i(\diag(a,d))$.

Note that each $γ∈S(F)_{\rm{rs}}$ satisfies $c(γ)=(1-N(a(γ)))/\overbar{b(γ)}$ and $d(γ)=-\overbar{a(γ)}b(γ)/\overbar{b(γ)}$. Let $b_0$ be such that $-ab_0/\overbar{b_0}=d$. Then if $γ$ is close to $\diag(a,d)$, then $b(γ)$ is close to $F^\times b_0$ and hence $η(b(γ)/\overbar{b(γ)})=η(b_0/\overbar{b_0})$.

Expanding the definition of transfer, we see that we need to solve the following system of equations of functions on $X$:
\begin{align*}
&η(1-N(a))η(b_0/\overbar{b_0})A_0(a,d,b)+A_1(a,d,c)\\
=& \begin{cases}C_0(a,d) & \tif\ η(1-N(a))=1\\
                C_1(a,d) & \tif\ η(1-N(a))=-1\end{cases}
\end{align*}
A solution is given by
\begin{equation}\label{solutiontransfer}A_0(a,d,b):=η(\overbar{b_0}/b_0)(C_0-C_1)/2 \tand A_1(a,d,c):=(C_0+C_1)/2.
\end{equation}

For the converse, let $f∈C^\infty_c(S(F))$ and fix $a,d∈U(1)(F)$. If a regular semi-simple $γ$ is near $\diag(a,d)$, then $ω(γ)η(b(γ))=η(1-N(a))η(b_0/\overbar{b_0})=\pm η(b_0/\overbar{b_0})$ as explained above. It follows that $ω(γ)\Orb_γ(f)$ takes only two values near $\diag(a,d)$, depending on $η(1-N(a))$. Denote these values by $C_0(a,d)$ and $C_1(a,d)$, meaning
$$ω(γ)\Orb_γ(f)=\begin{cases}C_0(a,d) & \tif 1-N(a) \text{ is a norm}\\
                             C_1(a,d) & \tif 1-N(a) \text{ is not a norm},\end{cases}$$
whenever $γ$ is near $B_0$. Consider the map $$Ψ_i:U_i(F)_{\rm{rs}}\niδ\mapsto ω(γ)\Orb_γ(f)$$ where $γ$ is a match for $δ$. We need to show that $Ψ_0$ and $Ψ_1$ are given by orbital integrals on $U_0(F)$ and $U_1(F)$.

Clearly there exist $(φ_0,φ_1)∈C^\infty_c(U_0(F))\times C^\infty_c(U_1(F))$ with $φ_i(\diag(a,d))=C_i(a,d)$ for all $a,d∈U(1)(F)$. Then $\Orb_δ(φ_i)=C_i(a,d)$ for regular semi-simple $δ$ in a neighborhood of $\diag(a,d)$. We now consider the difference $α_i(δ):=Ψ_i(δ)-\Orb_δ(φ_i)$. Denote its support by $K_i⊂U_i(F)_{\rm{rs}}$. It is open and compact. 

The function $δ\mapsto \Vol(U(1)(F)δ\cap K_i)$ is locally constant and $U(1)(F)$-invariant. For $δ∈K_i$ we define
$$\widetilde{α_i}(δ):=\begin{cases} 0 & \tif\ U(1)(F)δ\cap K_i = \emptyset\\
                                α_i(δ)/\Vol(U(1)(F)δ\cap K_i) & \rm{otherwise}.\end{cases}$$

Then $\Orb_δ(\widetilde{α_i})=α_i$. And hence $Ψ_i(δ)=\Orb_δ(φ_i+\widetilde{α_i})$.
\end{proof}

\section{The Arithmetic Transfer identity}
\subsection{The group of quasi-isogenies $G$}
In the introduction we defined $G⊂\Aut^0(\mbX^{(2)})$ to be the group of $\mcO_E$-linear quasi-isogenies which preserve the polarization. We can identify this group as follows. Consider the embedding $E→M_2(D)$ defined by the action of $\mcO_E$ on $\mbX^{(2)}$. It is given as
$$x \mapsto \begin{cases}\left(\begin{smallmatrix} x & 0 \\ 0 & \overbar{x} \end{smallmatrix}\right) & \tif E/F \text{ is ramified or if } i+j \text{ is even}\\
                \left(\begin{smallmatrix} x & 0 \\ 0 & x \end{smallmatrix}\right) & \text{otherwise}.\end{cases}$$

In the first case, let $\varpi∈D$ be an element with $-\varpi^2=ε_1$ and $\varpi a=\overbar{a}\varpi$ for all $a∈\mcO_E$. Then the $E$-linear quasi-endomorphisms of $\mbX^{(2)}$ are given by $\varpi M_2(E)\varpi^{-1}⊂M_2(D)$ where the notation means conjugation by $\diag(\varpi,1)$. In the second case, the centralizer of $E$ is $M_2(E)⊂M_2(D)$, and we let $\varpi∈E$ be such that $N(\varpi)=ε_0$. In any case, $G⊂\varpi M_2(E)\varpi^{-1}$ is the group of matrices $A$ such that $A^*A=\id_2$, where $*$ denotes the transposition and standard involution.

We now consider both $G$ and $GL_2(E)$ as subgroups of $GL_2(D)$. Then conjugation with $\diag(\varpi^{-1},1)$ defines an isomorphism
\begin{equation}\label{GUIso}
G\overset{\iso}{→}\begin{cases} U_1(F) & \tif E/F \text{ is ramified or if } i+j \text{ is even}\\
              U_0(F) & \text{otherwise.}\end{cases}\end{equation}

Let $S(F)_G⊂S(F)_{\rm{rs}}$ be the elements which match in $U_1(F)$ if $E/F$ is ramified or if $i+j$ is even. Otherwise let $S(F)_G$ be the elements matching in $U_0(F)$.
\begin{defn} An element $γ∈S(F)_G$ is said to match $g∈G$, if it matches a $δ\in U_i(F)$ which maps to $g$ under isomorphism \eqref{GUIso}.\end{defn}
A direct computation using formula \eqref{match} shows that $γ∈S(F)_G$ matches
\begin{align}\label{ultimateg}
g=\left(\begin{array}{cc}
\varpi a\varpi^{-1} & \varpi x\\
-\overbar{x}bε/\overbar{b}\cdot\varpi^{-1} & -\overbar ab/\overbar{b}
\end{array}\right),\end{align}
where $ε=ε_0$ or $ε=ε_1$ depending on the case and $N(x)=(1-N(a))/ε$. We complement Remark \ref{robustness} with the following corollary.

\begin{cor}
Let $γ∈S(F)_G$ match $g∈G$. Then $g$ is unique up to conjugation by $U(1)(F)$, independent of the chosen $ε_0,ε_1$ in the definition of the unitary groups.
\end{cor}
\begin{proof}
This is immediate from the previous formula. Namely let $t∈D$ have Norm $1-N(a)$. Then the product $\varpi x$ lies in $tU(1)(F)$, independently of $ε$.
\end{proof}

\subsection{The Arithmetic Fundamental Lemma}
We recall the AFL in the case $n=2$ in our terminology. This is done to illustrate similarities and differences to our arithmetic transfer identity as explained in the introduction. So only in this subsection, $E/F$ is unramified, $i=j=0$ and the extension of $η$ to $E^\times$ is given by $η(x)=(-1)^{v(x)}$.

Consider the formal $\mcO_F$-module $X_0\times \overbar{X}_0$ over $\Spf \mcO_{\breve E}$ and let $g∈\End^0(\mbX\times \overbar{\mbX})$ be a quasi-endomorphism. We define $\Def(g)⊂\Spf\mcO_{\breve E}$ to be the maximal closed subscheme to which $g$ deforms as endomorphism. We set $\Int(g):=\len_{\mcO_{\breve E}}\Def(g)$, possibly $\infty$. Then the Arithmetic Fundamental lemma for $n=2$ is the following theorem.

\begin{thm} \label{aflemma}\cite[Theorem 2.10]{zhang}
Let $E/F$ be unramified and assume that in the definition of the unitary groups $ε_0=1$ and $ε_1=π_F$. Assume that $γ∈S(F)_G$ matches $g$ in $G$. Let $1_{K}$ be the characteristic function of $K=GL_2(\mcO_E)\cap S(F)$. Then $\Def(g)$ is artinian and there is an equality
\begin{align}\label{afl}ω(γ)\dOrb_γ(1_K)=\Int(g)\cdot \log q.\end{align}
Furthermore, the function $1_{K}$ has transfer $(1_{U_0(\mcO_E)}, 0)$.
\end{thm}
\begin{proof}
Recall that $γ$ has the form
$$γ=\left(\begin{matrix}a & b\\ (1-N(a))/\overbar{b} & -\overbar{a}b/\overbar{b}\end{matrix}\right).$$
The condition that $γ$ matches in $U_1(F)$ is equivalent to $π_F(1-N(a))$ being a norm and hence to $v(1-N(a))$ being odd. In particular, $a∈\mcO_E$. We use formula \eqref{ultimateg} with $ε=π_F$ to compute $g$.

\emph{Calculation of the right hand side:}
We can write $g$ as a matrix
$$g=\left(\begin{array}{cc}g_1 & g_2 \\ g_3 & g_4\end{array}\right)$$
with all $g_i∈\End^0(\mbX)$. Deforming $g$ is equivalent to deforming all four entries separately since $\mbX\times \overbar{\mbX}$ is lifted factor wise. The entries $g_2$ and $g_3$ Galois-commute with the $\mcO_E$-action. So by Corollary \ref{corglobalhom} b), they do not lift to $\End(X_0)$. It follows that $\Def(g)$ is artinian.

The entries $g_1$ and $g_4$ lie in $\mcO_E$ and deform arbitrarily far. The entries $g_2,g_3$ both have valuation equal to
$$v_D\left(\varpi x\right)=1+v_D(x)=1+v_F((1-N(a))/π_F)=v_F(1-N(a)).$$
By Theorem \ref{thm1}, the length of the deformation locus equals $\frac{1}{2}\Big(1+v_F(1-N(a))\Big)$.

For the analytic side, we compute
\begin{align*}\dOrb_γ(1_K) =&\ 
-\log(q)\int_{F^\times} v(h)η(h) 1_K\left(\left(\begin{smallmatrix} a & b/h \\ h(1-N(a))/\overbar{b} & \ -\overbar{a}b/\overbar{b}\end{smallmatrix}\right)\right)dh\\
=&\ \log(q)\sum_{i=v(b)-v(1-N(a))}^{v(b)} (-1)^{i+1} i\end{align*}

Now we use that $v(1-N(a))$ is odd, so that the sum has an even number of summands to see
\begin{align*}\dOrb_γ(1_K)=&\ (-1)^{v(b)-v(1-N(a))}\log(q)\sum_{i=0}^{v(1-N(a))} (-1)^{i+1} i\\
                                    =&\ ω(γ)^{-1}\log(q)\frac{1+v(1-N(a))}{2}.\end{align*}
\end{proof}

\subsection{The Arithmetic Transfer Identity}
Now we fix two quasi-canonical lifts $X_i$ and $Y_j$ of levels $i$ and $j$, defined over a finite extension $A/\mcO_{\breve F}$ of ramification index $e$. The formal $\mcO_F$-module $X_i\times Y_j$ is then a deformation of $\mbX^{(2)}$. For a given quasi-homomorphism $g∈\End^0(\mbX^{(2)})$, we define $\Def(g)$ to be the maximal closed subscheme of $\Spf A$ to which $g$ deforms as endomorphism of $X_i\times Y_j$. Let $\Int(g):=\len_{\mcO_{\breve E}}\Def(g)$ as above.

Recall that $K_{i,j}$ denotes the stabilizer of $\mcO_i\oplus \mcO_j$ in $U_k(F)$, where $k=0$ if $E/F$ is ramified or if $i+j$ is even and $k=1$ otherwise. Let $1_{K_{i,j}}$ be its characteristic function. The following theorems are our main results. Note that Theorem \ref{firstthmB} follows for $E/F$ unramified and $i=j=0$ from the AFL (Theorem \ref{aflemma}).

\begin{thm}\label{firstthmB}
There exists a function $f∈C^\infty_c(S(F))$ which is a transfer of 
$$\begin{cases} (e\cdot 1_{K_{i,j}},0) & \tif E/F \text{ ramified or if } i+j \text{ is even}\\
                (0,e\cdot 1_{K_{i,j}}) & \tif E/F \text{ unramified and if } i+j \text{ is odd}\end{cases}$$
with the following property. For any $γ∈S(F)_G$ matching $g∈G$, the length $\Int(g)$ is finite and there is an equality
\begin{align}\label{firststmtB}ω(γ)\dOrb_γ(f)=\Int(g)\cdot \log(q).\end{align}
\end{thm}

\begin{thm}\label{secondthmB}
For every function $f∈C^\infty_c(S(F))$ which is a transfer of
$$\begin{cases} (e\cdot 1_{K_{i,j}},0) & \tif E/F \text{ ramified or if } i+j \text{ is even}\\
                (0,e\cdot1_{K_{i,j}}) & \tif E/F \text{ unramified and if } i+j \text{ is odd},\end{cases}$$
there exists a function $f_{\rm{corr}}∈C^\infty_c(S(F))$ such that for any $γ∈S(F)_G$ matching $g∈G$:
\begin{equation}\label{secondstmtB}ω(γ)\dOrb_γ(f)=\Int(g)\cdot \log(q) + ω(γ)\Orb_γ(f_{\rm{corr}}).\end{equation}
\end{thm}

\emph{Proof of Theorem \ref{firstthmB}.}
Let us assume Theorem \ref{secondthmB}. By Theorem \ref{existtransfer}, there exists a function $f_0$ which has the correct transfer. Let $f_{\rm{corr}}$ yield the correction term in Theorem \ref{secondthmB}. By Lemma \ref{basicsonorbitalintegrals} b), there exists a function $f_1$ with $\dOrb_γ(f_1)=\Orb_γ(f_{\rm{corr}})$ for all $γ∈S(F)_{\rm{rs}}$. By Lemma \ref{basicsonorbitalintegrals} a), $f_1$ has transfer $(0,0)$. It follows that $f:=f_0+f_1$ is a function as in Theorem \ref{firstthmB}. $\hfill\qed$

\emph{Proof of Theorem \ref{secondthmB}.}
Let $f$ be a function as in the theorem. Denote by $β$ the characteristic function of the set
$$\mcO_i^\times\times\mcO_j^\times\times(E^\times/F^\times)⊂X.$$
By formula \eqref{solutiontransfer} in the proof of Theorem \ref{existtransfer}, the germ expansion of $f$ is given by $A_0=\pm \frac{1}{2}η(\overbar{b}/b)e\cdot β$ and $A_1=\frac{1}{2}e\cdot β$. The sign is $+$ precisely if $E/F$ is ramified or if $i+j$ is even.

By the formula in Corollary \ref{charactderivorbints}, this determines the leading terms in the germ expansion of $\dOrb_γ(f)$. Let $γ=\left(\begin{smallmatrix}a & b\\ c& d\end{smallmatrix}\right)$ be near $B_0$. Then
$$\dOrb_γ(f)=\frac{e\log(q)}{2}\Big(\mp η(\overbar b)v(b)+ η(c)^{-1}v(c)\Big)β(γ)+R(γ),$$
where $R(γ)$ is constant part of the germ expansion.
The sign change comes from the derivative of $η_s(b)=q^{-sv(b)}η(b)$. We now use that $c=(1-N(a))/\overbar{b}$ and \eqref{trafa} to get
$$ω(γ)\dOrb_γ(f)=\frac{e\log(q)}{2}\cdot\Big(\mp η(1-N(a))v(b)+ v(c)\Big)β(γ)+ω(γ)R(γ).$$
Finally if we choose $γ∈S(F)_G$, then $η(1-N(a))=-1$ precisely in the first case of the theorem. So for $γ∈S(F)_G$ near $B_0$, the left hand side of \eqref{secondstmtB} equals
$$ω(γ)\dOrb_γ(f)=\frac{e\log(q)}{2}v(1-N(a))β(γ)+ω(γ)R(γ).$$

Now let us turn to the geometric side. Let $g∈G$ be a match for some $γ∈S(F)_G$. We write $g$ as a matrix with entries $g_1,g_2,g_3,g_4$. With the same arguments as in the proof of Theorem \ref{aflemma}, $\Int(g)$ is finite. We consider the difference $o(γ):=ω(γ)\dOrb_γ(f)-\Int(g)\log(q)$ for $γ∈S(F)_G$.

\emph{Claim: The function $o(γ)$ (on $S(F)_G$) is constant near $B_0$.}\smallskip\\
First note that $v_D(g_2)$ and $v_D(g_3)$ will tend to infinity as $γ$ approaches $B_0$. This follows from formula \eqref{ultimateg}. Now fix $\diag(a,d)∈B_0$. We distinguish two cases.

If $(a,d)\notin\mcO_i^\times\times\mcO_j^\times$, then $g_1$ and $g_4$ will not deform arbitrarily far for $g$ near $\diag(a,d)$. In fact, the length of the maximal subscheme of $\Spf A$ to which they deform is computed in Theorem \ref{thm1}. This length equals $\Int(g)$ for all $γ$ near $\diag(a,d)$.

If instead $(a,d)∈\mcO_i^\times\times\mcO_j^\times$, then $g_1$ and $g_4$ deform arbitrarily far for all $γ$ near $\diag(a,d)$. By Theorem \ref{thm1}, $\Int(g)$ grows linearly in $v(1-N(a))$ for $γ$ approaching $\diag(a,d)$. More precisely, $\Int(g)-\frac{e}{2}v(1-N(a))$ is constant near $\diag(a,d)$. It follows that $o(γ)$ is constant near $\diag(a,d)$, which proves the claim.

We define $A_0(a,d)$ to be the value of $o(γ)$ near $\diag(a,d)$ and extend it to a smooth function on $X$. By definition, $ω^{-1}o$ has a germ expansion associated to $A_0$ and $A_1:=0$ when restricted to $γ∈S(F)_G$, i.e.
$$ω(γ)^{-1}o(γ)=η(c(γ))^{-1}A_0(a(γ),d(γ)).$$

By the second part of Corollary \ref{charactorbints}, there exists a function $f_{\rm{corr}}$ satisfying
$$ω(γ)\Orb_γ(f_{\rm{corr}})=o(γ)\ \ \forall γ∈S(F)_G.$$
This finishes the proof of Theorem \ref{secondthmB}.$\hfill\qed$

We conclude with a result about the naturality of Theorems \ref{firstthmB} and \ref{secondthmB}.

\begin{prop}\label{ambiguity}
The function $e\cdot 1_{K_{i,j}}$ in the theorems can be replaced by any function $g_0$ such that $g_0-e\cdot 1_{K_{i,j}}$ vanishes in a neighborhood of $U_k(F)\setminus U_k(F)_{\rm{rs}}$.
\end{prop}
\begin{proof}
This is clear, since the proof of Theorem \ref{secondthmB} only depended on $e\cdot 1_{K_{i,j}}|_{B_0}$.
\end{proof}

\end{document}